\documentclass{article}
\usepackage[affil-it]{authblk}
\usepackage[utf8]{inputenc}
\usepackage{amsmath, amsfonts, amsthm, amssymb, latexsym}
\usepackage{geometry}
\usepackage{enumerate}
\usepackage{multirow}
\usepackage{rotating}
\usepackage{graphicx,color}
\usepackage{float}
\usepackage{mathrsfs}

\usepackage{xargs}
\usepackage[pdftex,dvipsnames]{xcolor}

\usepackage{hyperref}
\hypersetup{colorlinks=true, pdfstartview=FitV,linkcolor=blue!70!black,citecolor=red!70!black, urlcolor=green!60!black}
\definecolor{labelkey}{rgb}{0.6,0,0}

\usepackage[colorinlistoftodos,prependcaption,textsize=small]{todonotes}
\newcommandx{\change}[2][1=]{\todo[#1]{#2}}
\newcommandx{\unsure}[2][1=]{\todo[linecolor=red,backgroundcolor=red!25,bordercolor=red,#1]{#2}}
\newcommandx{\rmk}[2][1=]{\todo[linecolor=blue,backgroundcolor=blue!25,bordercolor=blue,#1]{#2}}
\newcommandx{\info}[2][1=]{\todo[linecolor=OliveGreen,backgroundcolor=OliveGreen!25,bordercolor=OliveGreen,#1]{#2}}
\newcommandx{\improvement}[2][1=]{\todo[linecolor=Plum,backgroundcolor=Plum!25,bordercolor=Plum,#1]{#2}}
\newcommandx{\thiswillnotshow}[2][1=]{\todo[disable,#1]{#2}}

\newtheorem{thm}{Theorem}[section]

\newtheorem{prop}[thm]{Proposition}
\newtheorem{cor}[thm]{Corollary}

\theoremstyle{definition}

\theoremstyle{remark}

\title{An inverse problem for the fractional Allen-Cahn equation}

\author{Li Li \thanks{lili19940301@mail.tsinghua.edu.cn}}
\affil{Yau Mathematical Sciences Center, Tsinghua University, Beijing, China}

\date{}

\begin{document}
	
	\maketitle
	
	\noindent \textbf{ABSTRACT.}\, We study an inverse problem for the fractional Allen-Cahn equation. Our formulation and arguments rely on the asymptotics for the fractional equation
	and unique continuation properties.
	
	\section{Introduction}
	We consider the following fractional Allen-Cahn equation
	\begin{equation}\label{fracAC}
		\left\{
		\begin{aligned}
			(-\Delta)^s u_{\epsilon}+\frac{1}{\epsilon^{2s}}W'(u_{\epsilon}) &= f\quad \,\,\, \mathrm{in}\,\, \Omega,\\
			u_{\epsilon}&= g\quad \,\mathrm{in}\,\, \mathbb{R}^n\setminus\Omega.\\
		\end{aligned}
		\right.
	\end{equation}
	Here $\epsilon> 0$ is a small parameter. We require that the potential
	$W\in C^2(\mathbb{R}; [0,\infty))$ satisfies 
	\begin{equation}\label{Wpm1}
		W> 0\,\,\mathrm{in}\,\,(-1, 1),\qquad\{W= 0\}=\{\pm 1\},\qquad W'(\pm 1)= 0,\qquad W''(\pm 1)> 0
	\end{equation}
	and there exist constants $p\in [2, \infty)$, $C> 0$ s.t.
	\begin{equation}\label{pest}
		\frac{1}{C}(|t|^{p-1}-1)\leq |W'(t)|\leq C(|t|^{p-1}+1).
	\end{equation}
	
	In recent years, there have been many studies on the fractional equation in (\ref{fracAC}), which can be viewed as the fractional analogue of the classical elliptic Allen-Cahn (or scalar Ginzburg-Landau) equation. Most of them are motivated by physical problems related to dislocation dynamics or stochastic Ising models arsing from statistical mechanics (see \cite{imbert2009phasefield} and the references therein). 
	
	In this paper, we suppose that $s\in (0, \frac{1}{2})$, $\epsilon_k\to 0^+$ and $|g|= 1$ a.e. in $\mathbb{R}^n\setminus\Omega$. We use $u_k$ to denote a solution of (\ref{fracAC}) corresponding to the parameter $\epsilon_k$. Under certain assumptions, it has been shown that there exists a sequence $\{u_k\}$ s.t.
	$$u_k\to \chi_{E_*}-\chi_{\mathbb{R}^n\setminus E_*}$$ in $\Omega$ for some set $E_*$ satisfying the $2s$-mean curvature of the surface $\partial E_*\cap \Omega$ coincides with $f$ up to a constant only depending on $n$ and $s$. Here $\chi$ denotes the standard characteristic function.
	We will provide an accurate statement of this asymptotic behavior later in Proposition \ref{convTh}.
	
	Our main goal here is to formulate an inverse problem based on the convergence result above. The following theorem is our main result in this paper.
	
	\begin{thm}\label{mainTh}
		We assume that the potential $W$ is analytic on $\mathbb{R}$. Let $V\subset \Omega$ be a nonempty open set outside the support of $f$.
		Then we have the following statements.
		
		(i)\, $W$, $f$ and $\partial E_*\cap \Omega$ can be determined from the knowledge of
		$\{u_k|_V, (-\Delta)^s u_k|_V\}$.
		
		(ii)\,  $f$ and $\partial E_*\cap \Omega$ can be determined from the knowledge of $W$ and $\{(-\Delta)^s u_k|_V\}$.
	\end{thm}
	
	We remark that our assumption on the analyticity of the potential is reasonable since this at least includes the most classical example 
	$$W(t)= \frac{(1- t^2)^2}{4}.$$
	
	We will provide an accurate statement of this theorem later in Theorem \ref{Th}.
	
	\subsection{Connection with earlier literature}
	The rigorous mathematical study of Calder\'on type inverse problems for the space-fractional equations was initiated in \cite{ghosh2020calderon}. In the seminal paper \cite{ghosh2020calderon}, the authors have shown that the inverse problem can be solved in a relatively straightforward way based on nonlocal phenomena such as the unique continuation property of the fractional Laplacian and the associated Runge approximation property.
	
	We refer readers to \cite{ghosh2020uniqueness, ruland2020fractional, covi2022higher, feizmohammadi2021fractional, feizmohammadi2024calder}
	and the references therein for extended results and various variants in this direction.
	More specifically, we mention that Calder\'on type inverse problems for the fractional nonlinear equations have been studied in \cite{lai2022inverse, lai2023inverse, li2021inverse, li2025inverse}. The approaches in these papers rely on either the first order linearization technique or the multiple-fold linearization technique combined with nonlocal features of the models.
	
	In this paper, our approach will still rely on the unique continuation property of the fractional Laplacian. The main novelty is to exploit the asymptotics for the fractional Allen-Cahn equation in formulating and solving the inverse problem. Compared with the standard linearization technique, this approach has a more explicit geometric interpretation, which sets our model apart from
	the earlier studies mentioned above.
	
	\subsection{Organization}
	The rest of this paper is organized in the following way. In Section 2, we will summarize the preliminary knowledge. In Section 3, we will focus on the existence of solutions of (\ref{fracAC}). In Section 4, we will focus on the unique continuation of the fractional Laplacian in the Sobolev space $\hat{H}^s(\Omega)$. In Section 5, we will focus on asymptotics for the fractional equation (\ref{fracAC}). In Section 6, we will rigorously formulate our inverse problem and provide the proof. In Section 7, we will partially extend the main theorem in the setting of multi-phase transitions.
	\medskip
	
	\noindent \textbf{Acknowledgments.} The author would like to thank Professor Gunther Uhlmann for helpful discussions.
	
	\section{Preliminaries}
	
	Throughout this paper, we use $s\in (0, \frac{1}{2})$ to denote the fractional power, and we use $\Omega$ to denote a bounded domain with smooth boundary. 
	
	\subsection{Function spaces}
	The energy functional associated with (\ref{fracAC}) will be denoted by
	$$\mathcal{F}_\epsilon(u, \Omega):=
	\mathcal{E}_\epsilon(u, \Omega)-\int_\Omega fu\,\mathrm{d}x.$$
	Here $\mathcal{E}_\epsilon(u, \Omega)$ is the energy functional associated with the homogeneous equation. More precisely,
	$$\mathcal{E}_\epsilon(u, \Omega):=
	\mathcal{E}(u, \Omega)+ \frac{1}{\epsilon^{2s}}\int_\Omega W(u)\,\mathrm{d}x,$$
	which consists of Sobolev and 
	potential parts, where the Sobolev energy is given by
	$$\mathcal{E}(u, \Omega):= \frac{c_{n,s}}{4}(\int_\Omega\int_\Omega+ 2\int_\Omega\int_{\mathbb{R}^n\setminus\Omega})\frac{|u(x)- u(y)|^2}{|x-y|^{n+2s}}\mathrm{d}x\mathrm{d}y.$$
	The positive constant $c_{n,s}$ only depends on $n$ and $s$. Here we only consider the $\Omega$-contribution in $H^s$-norm so we omit the set $(\mathbb{R}^n\setminus\Omega)\times (\mathbb{R}^n\setminus\Omega)$ when we integrate. 
	
	We use $\tilde{H}^s(\Omega)$ to denote the closure of $C^\infty_c(\Omega)$ in  $H^s(\mathbb{R}^n)$,
	Here $H^s(\mathbb{R}^n)$ is the standard $L^2$-based Sobolev space $W^{s,2}(\mathbb{R}^n)$. Since $\partial\Omega$ is smooth, it is well-known that $u\in \tilde{H}^s(\Omega)$ if and only if
	$u\in H^s(\mathbb{R}^n)$ with $\mathrm{supp}\,u\subset \overline{\Omega}$.
	We write
	$$\hat{H}^s(\Omega):= \{u\in L^2_{loc}(\mathbb{R}^n): \mathcal{E}(u, \Omega)< \infty\},$$
	which is a Hilbert space equipped with the norm
	$$||u||_{\hat{H}^s(\Omega)}:= (||u||^2_{L^2(\Omega)}+ \mathcal{E}(u, \Omega))^\frac{1}{2}.$$
	We also have $$H^s_{loc}(\mathbb{R}^n)\cap L^\infty(\mathbb{R}^n)\subset \hat{H}^s(\Omega)$$ (see Section 2 in \cite{millot2019asymptotics}).
	For $g\in \hat{H}^s(\Omega)$, we use $H^s_g(\Omega)$ to denote the affine space $g+ \tilde{H}^s(\Omega)$.
	
	The fractional Laplacian $(-\Delta)^s$ can 
	be defined by the pairing 
	\begin{equation}\label{fracLpair}
		\langle (-\Delta)^su, \varphi \rangle:= \frac{c_{n,s}}{2}(\int_\Omega\int_\Omega+ 2\int_\Omega\int_{\mathbb{R}^n\setminus\Omega})\frac{(u(x)- u(y))(\varphi(x)- \varphi(y))}{|x-y|^{n+2s}}\mathrm{d}x\mathrm{d}y
	\end{equation}
	for $u, \varphi\in \hat{H}^s(\Omega)$. If $\varphi$ is restricted to $\tilde{H}^s(\Omega)$, then $(-\Delta)^su$ should be viewed as an element in $H^{-s}(\Omega)$ (the dual space of $\tilde{H}^s(\Omega)$), and it is not difficult to see that this definition coincides with the pointwise definition
	$$ (-\Delta)^su(x):= c_{n,s}\,\mathrm{p.v.}\int_{\mathbb{R}^n}\frac{u(x)- u(y)}{|x-y|^{n+2s}}\mathrm{d}y$$
	for smooth bounded $u$.
	
	\subsection{Concepts in geometry}
	
	The fractional power $s\in (0, \frac{1}{2})$ ensures that the fractional $2s$-perimeter
	$$P_{2s}(U)= P_{2s}(U, \mathbb{R}^n):= 
	\int_{\mathbb{R}^n\setminus U}\int_{U}\frac{1}{|x-y|^{n+2s}}\mathrm{d}x\mathrm{d}y< \infty$$
	for every Lipschitz bounded domain $U$ (see Remark 1.4 in \cite{cinti2019quantitative}).
	It is well-known that 
	$$\frac{1}{\omega_{n-1}}\lim_{s\to \frac{1}{2}^-}(1-2s)P_{2s}(E)= P(E):= \mathcal{H}^{n-1}(\partial E)$$
	whenever $E$ is an open set with smooth boundary. Here $\omega_n$ denotes the volume of the $n$-dimensional unit ball and $\mathcal{H}^{n-1}$ denotes the Hausdorff $(n-1)$-dimensional measure. Hence, $P_{2s}$ can be naturally viewed as a fractional analogue of the classical perimeter $P$.
	
	The fractional $2s$-perimeter in $\Omega$ of a set $E\subset \mathbb{R}^n$ is defined by 
	\begin{equation}\label{farc2sP}
		P_{2s}(E, \Omega):= \frac{1}{2c_{n,s}}\mathcal{E}(\chi_E-\chi_{\mathbb{R}^n\setminus E}, \Omega)
	\end{equation}
	$$= (\int_{E\cap \Omega}\int_{E^c\cap \Omega}
	+ \int_{E\cap \Omega}\int_{E^c\setminus \Omega}+ \int_{E\setminus \Omega}\int_{E^c\cap \Omega})\frac{1}{|x-y|^{n+2s}}\mathrm{d}x\mathrm{d}y.$$
	Here $\chi_E$ is the characteristic function of $E$. 
	We consider the first variation of the $2s$-perimeter
	$$\delta P_{2s}(E, \Omega)[X]:= \frac{\mathrm{d}}{\mathrm{d}t}P_{2s}(\phi_t(E), \Omega)\vert_{t=0}.$$
	Here $\{\phi_t\}$ denotes the flow generated by the vector field $X\in C^1_c(\Omega; \mathbb{R}^n)$. We can compute that
	$$\delta P_{2s}(E, \Omega)[X]= \int_{\partial E\cap \Omega} H^{(2s)}_{\partial E}(x) X\cdot \nu_E\,\mathrm{d}\mathcal{H}^{n-1}$$
	(see the detailed computation in Section 6 in \cite{figalli2015isoperimetry}) where $\nu_E$ is the unit outer normal along $\partial E$ and the nonlocal $2s$-mean curvature is given by
	$$H^{(2s)}_{\partial E}(x):= -\mathrm{p.v.}\int_{\mathbb{R}^n}
	\frac{\chi_E(y)-\chi_{\mathbb{R}^n\setminus E}(y)}{|x-y|^{n+2s}}\mathrm{d}y.$$
	
	We say that $\partial E\cap \Omega$ is a nonlocal $2s$-minimal surface in $\Omega$ provided that $H^{(2s)}_{\partial E}= 0$ along $\partial E\cap \Omega$. Clearly, if $E$ minimizes its $2s$-perimeter in $\Omega$, i.e.
	$$P_{2s}(E, \Omega)\leq P_{2s}(F, \Omega)\qquad \forall\, F\subset \mathbb{R}^n\quad\mathrm{s.t.}\quad F\setminus\Omega= E\setminus\Omega,$$
	then $\partial E\cap \Omega$ is a nonlocal $2s$-minimal surface. 
	
	Given a general function $f$ in $\Omega$, the nonlocal $2s$-mean curvature equation
	$$H^{(2s)}_{\partial E}= \frac{1}{c_{n, s}}f$$
	along $\partial E\cap \Omega$ should be interpreted as the weak formulation
	$$\delta P_{2s}(E, \Omega)[X]=
	\frac{1}{c_{n, s}}\int_{E\cap\Omega}\mathrm{div}(fX)\,\mathrm{d}x,\qquad X\in C^1_c(\Omega; \mathbb{R}^n).$$
	
	We remark that the fractional energy $\mathcal{E}_\epsilon(u, \Omega)$ has also been studied when $s\in [\frac{1}{2}, 1)$, but in this case its $\Gamma$-convergence leads to the classical minimal surface functional instead of the fractional one (see the dichotomy between the two cases $s\geq 1/2$ and $s < 1/2$ in \cite{savin2012gamma}). 
	
	\section{The forward problem}
	
	Given $g\in \hat{H}^s(\Omega)$ and $f\in L^\infty(\Omega)$, we say that $u$ is a (weak) solution of (\ref{fracAC}) provided that $u$ is a critical point of the energy functional $\mathcal{F}_\epsilon(u, \Omega)$ i.e.
	$$\frac{\mathrm{d}}{\mathrm{d}t}\mathcal{F}_\epsilon(u+t\varphi, \Omega)\vert_{t=0}= 0,\qquad \varphi\in \tilde{H}^s(\Omega).$$
	Our goal here is to show the existence of a solution. As mentioned in Remark 3.1 in \cite{millot2019asymptotics}, it suffices to solve the corresponding minimization problem based on the variational approach. For self-completeness, we will provide more details below.
	
	\begin{prop}\label{exist}
		Suppose $g\in \hat{H}^s(\Omega)\cap L^p(\Omega)$.
		Then there exists a minimizer of $\mathcal{F}_\epsilon(u, \Omega)$ in 
		$$H:= H^s_g(\Omega)\cap L^p(\Omega).$$
	\end{prop}
	
	\begin{proof}
		It suffices to study $\mathcal{F}_\epsilon(u, \Omega)$ for $\epsilon= 1$.
		
		In fact, for $u\in H$, we can apply the fractional Poincaré inequality
		$$\langle (-\Delta)^s v, v\rangle\geq c||v||^2_{L^2(\Omega)}$$
		to $v:= u- g\in \tilde{H}^s(\Omega)$. We also apply Cauchy-Schwarz inequality to $\int_\Omega fu$ and estimate the potential energy based on (\ref{pest}). Then we obtain
		$$\mathcal{F}_1(u, \Omega)\geq 
		C_1(||u||^2_{\hat{H}^s(\Omega)}+ ||u||^p_{L^p(\Omega)})- C_2$$
		for $u\in H$. Here $C_1, C_2$ are positive constants. $C_2$ relies on both $f$ and $g$.
		
		Now suppose that $\{u_k\}\subset H$ and $u_k\rightharpoonup u$ weakly in 
		$\hat{H}^s(\Omega)\cap L^p(\Omega)$. Then the weak convergence in $L^p(\Omega)$ ensures that
		$$\int_\Omega fu_k\,\mathrm{d}x\to
		\int_\Omega fu\,\mathrm{d}x.$$
		Also $u_k-u\rightharpoonup 0$ weakly in $\tilde{H}^s(\Omega)$ and the compact embedding $\tilde{H}^s(\Omega)\hookrightarrow L^2(\Omega)$ ensure that there exists a subsequence (still denoted by $\{u_k\}$) s.t.
		$u_k\to u$ in $L^2(\Omega)$, and thus we can further take a subsequence s.t. $u_k\to u$ pointwise. Then we obtain 
		$$\mathcal{E}_1(u, \Omega)\leq \liminf_{k\to\infty}\mathcal{E}_1(u_k, \Omega)$$
		by Fatou’s lemma. Hence, we have verified 
		$$\mathcal{F}_1(u, \Omega)\leq \liminf_{k\to\infty}\mathcal{F}_1(u_k, \Omega).$$
		
		Now we conclude that a minimizer does exist in $H$ since both the coerciveness and the lower semi-continuity have been verified.
	\end{proof}
	
	We remark that a minimizer actually provides a stable solution of (\ref{fracAC}), which is not only a critical point of $\mathcal{F}_\epsilon(u, \Omega)$ but also satisfies
	$$\frac{\mathrm{d}^2}{\mathrm{d}t^2}\mathcal{F}_\epsilon(u+t\varphi, \Omega)\vert_{t=0}\geq 0,$$
	i.e.
	$$\frac{c_{n,s}}{2}(\int_\Omega\int_\Omega+ 2\int_\Omega\int_{\mathbb{R}^n\setminus\Omega})\frac{|\varphi(x)- \varphi(y)|^2}{|x-y|^{n+2s}}\mathrm{d}x\mathrm{d}y+ \frac{1}{\epsilon^{2s}}\int_\Omega W''(u)\varphi^2\,\mathrm{d}x\geq 0.$$
	This actually leads to some control of the potential energy by the Sobolev energy (see \cite{cabre2021stable} for details).
	
	\section{The unique continuation property}
	Now we state the unique continuation property of the fractional Laplacian for $u\in \hat{H}^s(\Omega)$.
	
	\begin{prop}\label{UCP}
		Suppose $u\in \hat{H}^s(\Omega)\cap L^\infty(\mathbb{R}^n)$. Let $U\subset \Omega$ be nonempty and open. If $$(-\Delta)^su= u= 0\quad\text{in}\,\,U,$$
		then $u= 0$ in $\mathbb{R}^n$.
	\end{prop}
	
	For $u\in H^s(\mathbb{R}^n)$, the unique continuation property above has been established in Theorem 1.2 in \cite{ghosh2020calderon} (see \cite{ruland2020fractional} and \cite{ghosh2020uniqueness} for its quantitative and constructive versions). For a general $u\in \hat{H}^s(\Omega)$, we cannot apply the theorem in \cite{ghosh2020calderon} directly, but fortunately the arguments in its proof can still be applied here based on the regularity results for the associated extension problem established in \cite{millot2019asymptotics}.
	
	\begin{proof}
		We consider the Caffarelli-Silvestre extension problem 
		\begin{equation*}
			\left\{
			\begin{aligned}
				\mathrm{div}(y^{1-2s}\nabla \tilde{u})&= 0\qquad\,\,\, \text{in}\,\,\mathbb{R}^{n+1}_+=\{(x, y): x\in \mathbb{R}^n, y> 0\},\\
				\tilde{u}&= u\qquad \,\,\text{on}\,\,\partial\mathbb{R}^{n+1}_+=\mathbb{R}^n\times\{0\}.\\
			\end{aligned}
			\right.
		\end{equation*}
		For $u\in \hat{H}^s(\Omega)$, (2.9) in \cite{millot2019asymptotics} provides the explicit formula for the solution $\tilde{u}$,
		and Lemma 2.10 in \cite{millot2019asymptotics} ensures that $\tilde{u}$ belongs to the weighted Sobolev space
		$$H^1_{loc}(\mathbb{R}^{n+1}_+\cup\Omega, |y|^{1-2s}\mathrm{d}\tilde{x})
		:= \{\tilde{u}\in  L^2_{loc}(\mathbb{R}^{n+1}_+\cup\Omega, |y|^{1-2s}\mathrm{d}\tilde{x}): \nabla\tilde{u}\in  L^2_{loc}(\mathbb{R}^{n+1}_+\cup\Omega, |y|^{1-2s}\mathrm{d}\tilde{x})\}.$$
		Here we write $\tilde{x}:= (x, y)\in \mathbb{R}^n\times \mathbb{R}$, and we identify $\Omega$ with $\Omega\times \{0\}$. Lemma 2.12 in \cite{millot2019asymptotics} ensures that $(-\Delta)^su(x)$ coincides with
		$$-\lim_{y\to 0^+}y^{1-2s}\partial_y\tilde{u}(x, y)$$
		up to a constant multiplicative factor $c_s$
		whenever $u\in \hat{H}^s(\Omega)$ and $x\in \Omega$. Moreover, (2.11) in \cite{millot2019asymptotics} ensures that  $\tilde{u}\in L^\infty(\mathbb{R}^{n+1}_+)$ whenever $u\in L^\infty(\mathbb{R}^n)$.

		Hence, the conditions required in Proposition 2.2 in \cite{ruland2015unique} are satisfied based on the assumption on $u$ so we obtain $\tilde{u}= 0$ in $B\cap \mathbb{R}_+^{n+1}$, where $B$ can be any open ball in $\mathbb{R}^{n+1}$ s.t. $B\cap \partial\mathbb{R}^{n+1}_+\subset U$. Then the classical theory of the elliptic equation with real-analytic coefficients ensures that $\tilde{u}$ is analytic in $\mathbb{R}_+^{n+1}$ so we conclude that 
		$\tilde{u}= 0$ in $\mathbb{R}_+^{n+1}$ and thus
		$u= 0$ in $\mathbb{R}^n$.
	\end{proof}
	
	Note that each constant $c$ belongs to $\hat{H}^s(\Omega)$ so we can apply the proposition above to $u-c$ instead of $u$ to obtain the following immediate corollary.
	
	\begin{cor}\label{UCPc}
		Suppose $u\in \hat{H}^s(\Omega)\cap L^\infty(\mathbb{R}^n)$. Let $U\subset \Omega$ be nonempty and open. If $$(-\Delta)^su= 0\quad \text{and}\quad u=c\quad\text{in}\,\,U,$$
		then $u= c$ in $\mathbb{R}^n$.
	\end{cor}
	
	If $f\in L^\infty(\Omega)$ and $g\in C^{0, 1}_{loc}(\mathbb{R}^n)\cap L^\infty(\mathbb{R}^n)$, then Proposition \ref{exist} ensures the existence of a solution $u\in H^s_g(\Omega)\cap L^p(\Omega)$ of (\ref{fracAC}). Moreover, by Theorem 3.9 and Corollary 3.10 (the maximum principle) in \cite{millot2019asymptotics}, we further have $u\in C(\mathbb{R}^n)\cap L^\infty(\mathbb{R}^n)$. Hence, the regularity assumption on $u$ in the statement is satisfied in this case.
	
	\section{Asymptotics for the fractional equation}
	Our precise formulation of the inverse problem will rely on the following convergence result (Theorem 5.1 in \cite{millot2019asymptotics})
	for the fractional Allen-Cahn equation.
	\begin{prop}\label{convTh}
		Suppose $s\in (0, \frac{1}{2})$ and  $\epsilon_k\to 0^+$. Let $\{g_k\}\subset C^{0, 1}_{loc}(\mathbb{R}^n)$ s.t.
		$$\sup_k ||g_k||_{L^\infty(\mathbb{R}^n\setminus\Omega)}< \infty$$ and $g_k\to g$ in $L^1_{loc}(\mathbb{R}^n\setminus\Omega)$ for some $g$ satisfying $|g|= 1$ a.e. in $\mathbb{R}^n\setminus\Omega$. Let $\{f_k\}\subset C^{0, 1}(\Omega)$ s.t.
		$$\sup_k \{\epsilon^{2s}_k||f_k||_{L^\infty(\Omega)}+ ||f_k||_{W^{1, q}(\Omega)}\}< \infty$$ 
		for some $q\in (\frac{n}{1+2s}, n)$ and $f_k \rightharpoonup f$
		weakly in $W^{1, q}(\Omega)$.
		
		Let $u_k\in H^s_{g_k}(\Omega)\cap L^p(\Omega)$ be a solution of 
		(\ref{fracAC}) corresponding to the parameter $\epsilon_k$, the exterior data $g_k$ and the source $f_k$. If
		$$\sup_k \mathcal{F}_{\epsilon_k}(u_k, \Omega)< \infty,$$
		then there exists a subsequence (still denoted by $\{u_k\}$)
		and a set $E_*\subset \mathbb{R}^n$ of finite $2s$-perimeter in $\Omega$ s.t. the following statements hold.
		
		(i)\, $E_*\cap \Omega$ is open and the nonlocal $2s$-mean curvature equation
		$$H^{(2s)}_{\partial E_*}= \frac{1}{c_{n, s}}f$$
		holds along $\partial E_*\cap \Omega$. In particular,
		$\partial E_*\cap \Omega$ is a nonlocal $2s$-minimal surface in $\Omega$ in the case $f= 0$.
		
		(ii)\, We have the convergence
		$$u_k\to u_*:= \chi_{E_*}-\chi_{\mathbb{R}^n\setminus E_*}\quad
		\mathrm{in}\,\,C^0_{loc}(\Omega\setminus \partial E_*).$$ 
		
		(iii)\, For each open $\Omega'$ compactly contained in $\Omega$, we have
		$$\mathcal{E}(u_k, \Omega')\to 2c_{n,s}P_{2s}(E_*, \Omega').$$
		
		(iv)\, For each open $\Omega'$ compactly contained in $\Omega$, we have the convergence
		$$f_k(x)- \frac{1}{\epsilon^{2s}_k}W'(u_k(x))\to
		(\frac{c_{n, s}}{2}\int_{\mathbb{R}^n}
		\frac{|u_*(x)-u_*(y)|^2}{|x-y|^{n+2s}}\mathrm{d}y)u_*(x)\quad
		\mathrm{in}\,\,H^{-s}(\Omega').$$
		
		(v)\, For each $\delta\in (-1, 1)$, compact set $K\subset \Omega$ and $r> 0$, we have
		$$\{u_k= \delta\}\cap K\subset \mathscr{T}_r(\partial E_*\cap \Omega),\qquad
		\partial E_*\cap K\subset \mathscr{T}_r(\{u_k= \delta\}\cap \Omega)$$
		for sufficiently large $k$. Here $\mathscr{T}_r$ represents the open tubular neighborhood of radius $r$.
	\end{prop}
	
	We remark that higher regularity results can be obtained in (ii). In fact, the convergence is in $C^{0, \alpha}_{loc}(\Omega\setminus \partial E_*)$ if $\sup_k ||f_k||_{L^{\infty}(\Omega)}< \infty$
	and in $C^{1, \alpha}_{loc}(\Omega\setminus \partial E_*)$ if $\sup_k ||f_k||_{C^{0, 1}(\Omega)}< \infty$ for every $\alpha\in (0, \beta_*)$, 
	where the constant $\beta_*= \beta_*(n, s)\in (0, 1)$ is given in Lemma 3.2 in \cite{millot2019asymptotics}.
	
	We also remark that (iv) relies on the formula
	$$(-\Delta)^su_*(x)= (\frac{c_{n, s}}{2}\int_{\mathbb{R}^n}
	\frac{|u_*(x)-u_*(y)|^2}{|x-y|^{n+2s}}\mathrm{d}y)u_*(x),$$
	which is based on the definition (\ref{fracLpair}), $u^2_*=1$ and the elementary identity
	$$(u_*(x)-u_*(y))(\varphi(x)-\varphi(y))= \frac{1}{2}|u_*(x)-u_*(y)|^2(u_*(x)\varphi(x)+ u_*(y)\varphi(y)).$$
	
	We need to mention that the proof of this convergence result relies on a careful analysis of the corresponding extension problem in $\mathbb{R}^{n+1}_+$ (see Section 4 in \cite{millot2019asymptotics}). If we only focus on a sequence of minimizers of the energy functional
	$\mathcal{F}_{\epsilon_k}(u, \Omega)$ (which are stable solutions), then the convergence result can be obtained in a much more straightforward way (see Section 2 in \cite{savin2012gamma}).
	
	\section{The inverse problem}
	We are ready to rigorously formulate our inverse problem. The following theorem will be an accurate version of Theorem \ref{mainTh}.
	
	\begin{thm}\label{Th}
		For $j=1, 2$, let $u^{(j)}_k\in H^s_{g^{(j)}_k}(\Omega)\cap L^p(\Omega)$ be a solution of 
		(\ref{fracAC}) corresponding to the potential $W^{(j)}$, the exterior data $g^{(j)}_k$ and the source $f^{(j)}_k$. We assume that $W^{(j)}$ satisfies (\ref{Wpm1}) and (\ref{pest}), and we assume that 
		$\{g^{(j)}_k\}$, $\{f^{(j)}_k\}$, $\{u^{(j)}_k\}$, $g^{(j)}$, $f^{(j)}$ and $E^{(j)}_*$ satisfy all the assumptions and conditions in the statement of Proposition \ref{convTh}.
		
		We further assume that $W^{(j)}$ is analytic on $\mathbb{R}$. We also further assume that 
		$E^{(j)}_*\cap \Omega\neq \Omega$ or $\emptyset$,
		$g^{(j)}_k\not\equiv 1$ or $-1$ in $\mathbb{R}^n\setminus\Omega$,
		$f^{(j)}_k\in C^1_c(\Omega)$, $\mathrm{supp}\, f^{(j)}_k\subset K$ for some compact set $K\subset \Omega$ for all $k$ and 
		$f^{(j)}_k\to f^{(j)}$ in $H^{-s}(\Omega)$. Let $V\subset \Omega$ be a nonempty open set outside $K$.
		
		(i)\, Suppose we know that
		$$u_k^{(1)}|_V= u_k^{(2)}|_V,\qquad (-\Delta)^su_k^{(1)}|_V= (-\Delta)^su_k^{(2)}|_V$$
		for all $k$. Then we have
		$$W^{(1)}= W^{(2)},\quad f^{(1)}= f^{(2)},\quad \partial E^{(1)}_*\cap \Omega= \partial E^{(2)}_*\cap \Omega,\quad P_{2s}(E^{(1)}_*, \Omega')= P_{2s}(E^{(2)}_*, \Omega')$$
		for each open $\Omega'$ compactly contained in $\Omega$.
		
		(ii)\, Suppose we know that
		$$W^{(1)}= W^{(2)},\qquad (-\Delta)^su_k^{(1)}|_V= (-\Delta)^su_k^{(2)}|_V$$
		for all $k$. Then we have
		$$f^{(1)}= f^{(2)},\quad \partial E^{(1)}_*\cap \Omega= \partial E^{(2)}_*\cap \Omega,\quad P_{2s}(E^{(1)}_*, \Omega')= P_{2s}(E^{(2)}_*, \Omega')$$
		for each open $\Omega'$ compactly contained in $\Omega$.
		
	\end{thm}
	
	We mention that here we do not consider exterior measurements of the fractional derivatives as in \cite{ghosh2020calderon} since $(-\Delta)^su$ is not defined outside $\Omega$ for a general $u\in \hat{H}^s(\Omega)$.
	
	\begin{proof}
		\textbf{Determine $W$ from the knowledge of 
			$\{u_k|_V, (-\Delta)^s u_k|_V\}$ in (i)}: The equation in (\ref{fracAC}) tells us that 
		\begin{equation}\label{fracACkj}
			(-\Delta)^s u^{(j)}_k+\frac{1}{\epsilon^{2s}_k}{W^{(j)}}'(u^{(j)}_k) = f^{(j)}
		\end{equation}
		in $\Omega$. By the assumptions in (i), we have
		$(-\Delta)^su_k^{(1)}|_V= (-\Delta)^su_k^{(2)}|_V$, $\mathrm{supp}\, f^{(j)}_k\subset K$ and $V\cap K = \emptyset$ implying that
		$${W^{(1)}}'(u_k)|_V= {W^{(2)}}'(u_k)|_V,\qquad u_k|_V:= u_k^{(1)}|_V= u_k^{(2)}|_V.$$
		
		By Proposition \ref{convTh} (ii), we can choose a small nonempty open $U_0\subset V$ away from $\partial E^{(j)}_*\cap \Omega$ s.t.
		$u_k\to 1$ or $-1$ in $U_0$. We claim that $u_k\not\equiv 1$ or $-1$ in any open $U\subset U_0$. Otherwise, we have $u_k\equiv 1$ or $-1$ in $U$ and $f^{(j)}_k= 0$ in $U$. Since ${W^{(j)}}'(\pm 1)= 0$, (\ref{fracACkj}) implies that
		$(-\Delta)^s u^{(j)}_k= 0$ in $U$. Then 
		by Corollary \ref{UCPc}, we have $u^{(j)}_k\equiv 1$ or $-1$ in $\mathbb{R}^n$, which contradicts the assumption that
		$g^{(j)}_k\not\equiv 1$ or $-1$ in $\mathbb{R}^n\setminus\Omega$. 
		Hence, by the continuity of $u_k$ in $U_0$, we can choose $\{x_k\}\subset U_0$ s.t. $x_k\to x_0\in U_0$, $u_k(x_k)\to 1$ or $-1$ but $u_k(x_k)\neq 1$ or $-1$ for each $k$. Based on
		$${W^{(1)}}'(u_k(x_k))= {W^{(2)}}'(u_k(x_k)),$$
		we conclude that ${W^{(1)}}'= {W^{(2)}}'$ and 
		thus ${W^{(1)}}= {W^{(2)}}=: W$ on $\mathbb{R}$ by the analytic continuation.

		\textbf{Determine $\{u_k|_{U_0}\}$ from the knowledge of $W$ and $\{(-\Delta)^s u_k|_V\}$ in (ii)}: By (\ref{fracACkj}) and the assumptions in (ii), we have 
		$$W'(u_k^{(1)})|_V= W'(u_k^{(2)})|_V.$$
		Note that $W''(\pm 1)> 0$ so $W'$ is strictly increasing near $\pm 1$. Recall that we have  
		$u^{(j)}_k\to 1$ or $-1$ in $U_0$. Suppose 
		$u^{(1)}_k\to -1$ and $u^{(2)}_k\to 1$ in $U_0$. Then the monotonicity of $W'$ near $1$ implies that
		$u^{(2)}_k= u^{(1)}_k+ 2$ in $U_0$ for large $k$. Since we have $(-\Delta)^s (u^{(2)}_k-
		u^{(1)}_k)= 0$ in $U_0$, by Corollary \ref{UCPc} we have $u^{(2)}_k= u^{(1)}_k+ 2$ in
		$\mathbb{R}^n$ and thus the only possibility is that $E^{(2)}_*\cap \Omega= \Omega$, $E^{(1)}_*\cap \Omega= \emptyset$ based on 
		Proposition \ref{convTh} (ii). However, this contradicts with the assumption that 
		$E^{(j)}_*\cap \Omega$ is nontrivial. Hence, we conclude that
		$\{u^{(1)}_k\}$ and $\{u^{(2)}_k\}$ converge both to $1$ or both to $-1$ in $U_0$
		Then we have $u_k^{(1)}= u_k^{(2)}$ in $U_0$ for large $k$ based on the monotonicity of $W'$ near $\pm 1$.
		
		\textbf{Determine $f$, $\partial E_*\cap \Omega$ and $P_{2s}(E_*, \Omega')$}: We have shown that in either (i) or (ii), we have
		$$u_k^{(1)}|_{U_0}= u_k^{(2)}|_{U_0},\qquad (-\Delta)^su_k^{(1)}|_{U_0}= (-\Delta)^su_k^{(2)}|_{U_0}$$
		for large $k$. Hence, by Proposition \ref{UCP} we conclude that $u^{(2)}_k= u^{(1)}_k=: u_k$ in $\mathbb{R}^n$ for large $k$. Then the identity $$\partial E^{(1)}_*\cap \Omega= \partial E^{(2)}_*\cap \Omega$$
		immediately follows from Proposition \ref{convTh} (ii). Another way to obtain this identity is to pick a value $\delta\in (-1, 1)$ and look at the level set $\{u^{(j)}_k= \delta\}$. Then the surface can be determined based on the locally uniform convergence of the level sets in Proposition \ref{convTh} (v). The identity
		$$P_{2s}(E^{(1)}_*, \Omega')= P_{2s}(E^{(2)}_*, \Omega')$$ immediately follows from Proposition \ref{convTh} (iii).
		
		We also conclude that $f^{(1)}= f^{(2)}$ based on Proposition \ref{convTh} (iv). More precisely, we have
		$$ f^{(j)}(x)= \lim_{k\to \infty}[\frac{1}{\epsilon^{2s}_k}W'(u_k(x))+
		(\frac{c_{n, s}}{2}\int_{\mathbb{R}^n}
		\frac{|u_*(x)-u_*(y)|^2}{|x-y|^{n+2s}}\mathrm{d}y)u_*(x)]$$
		in $H^{-s}(\Omega')$ for each open $\Omega'$ compactly contained in $\Omega$.
	\end{proof}
	
	\section{Fractional multi-phase transitions}
	Our goal here is to partially extend Theorem \ref{Th} in a more general setting.
	
	Here we consider the more general multiple-well potential $W$ instead of the double-well potential in previous sections. 
	More precisely, instead of (\ref{Wpm1}), now we require that
	\begin{equation}\label{WpmM}
		\{W= 0\}=\mathcal{Z}:= \{a_1,\cdots,a_m\},
	\end{equation}
	$W'= 0$ on $\mathcal{Z}$ and $W''> 0$ on $\mathcal{Z}$. Here $m\geq 2$. 
	Then (\ref{fracAC}) can be viewed as a  fractional analogue of the classical model studied in \cite{baldo1990minimal}.
	A basic example of a multiple-well potential is 
	$$W(t)= \prod^m_{j=1}(t-a_j)^2.$$ 
	In this case, we have $p= 2m$ in (\ref{pest}).
	
	Let $\mathfrak{E}= (E_1,\cdots,E_m)$ be a partition of $\mathbb{R}^n$. We define
	$\mathcal{P}^a_{2s}(\mathfrak{E}, \Omega)$
	by the expression
	$$\frac{1}{2}\sum^m_{i, j=1}(a_i- a_j)^2(\int_{E_i\cap \Omega}\int_{E_j\cap \Omega}
	+ \int_{E_i\cap \Omega}\int_{E_j\cap \Omega^c}+ \int_{E_i\cap \Omega^c}\int_{E_j\cap \Omega})\frac{1}{|x-y|^{n+2s}}\mathrm{d}x\mathrm{d}y,$$
	which is a natural generalization of the fractional $2s$-perimeter in $\Omega$ defined by (\ref{farc2sP}).
	
	We say that $\mathfrak{E}$ is a nonlocal $2s$-minimal partition in $\Omega$ provided that the first variation
	$$\delta \mathcal{P}^a_{2s}(\mathfrak{E}, \Omega)[X]:= \frac{\mathrm{d}}{\mathrm{d}t}\mathcal{P}^a_{2s}(\phi_t(\mathfrak{E}), \Omega)\vert_{t=0}$$
	vanishes for each flow $\{\phi_t\}$ generated by the vector field $X\in C^1_c(\Omega; \mathbb{R}^n)$. Here 
	$$\phi_t(\mathfrak{E}):= (\phi_t(E_1),\cdots, \phi_t(E_m))$$
	is still a partition of $\mathbb{R}^n$.
	
	The following proposition is a part of Theorem 1.1 in \cite{gabard2025fractional}, which can be viewed as an analogue of $\ref{convTh}$ in the setting of multi-phase transitions.
	
	\begin{prop}\label{convThM}
		Suppose $s\in (0, \frac{1}{2})$ and  $\epsilon_k\to 0^+$. Let $\{g_k\}\subset C^{0, 1}_{loc}(\mathbb{R}^n)$ s.t.
		$$\sup_k ||g_k||_{L^\infty(\mathbb{R}^n\setminus\Omega)}< \infty$$ and $g_k\to g$ for some $\mathcal{Z}$-valued $g$ a.e. in $\mathbb{R}^n\setminus\Omega$.
		
		Let $u_k\in H^s_{g_k}(\Omega)\cap L^p(\Omega)$ be a solution of
		(\ref{fracAC}) corresponding to the parameter $\epsilon_k$, the exterior data $g_k$ and the zero source. If
		$$\sup_k \mathcal{E}_{\epsilon_k}(u_k, \Omega)< \infty,$$
		then there exists a subsequence (still denoted by $\{u_k\}$)
		and a nonlocal $2s$-minimal partition $\mathfrak{E}_*$ of $\mathbb{R}^n$ with $\mathcal{P}^a_{2s}(\mathfrak{E}_*, \Omega)< \infty$ s.t. the following statements hold.
		
		(i)\, $E_{*_j}\cap \Omega$ is essentially open (i.e. it differs from an open set by a set of
		zero $n$-dimensional Lebesgue measure) and 
		$P_{2s}(E_{*_j}, \Omega)< \infty$.
		
		(ii)\, We have the convergence
		$$u_k\to u_*:= \sum_{j=1}a_j\chi_{E_{*_j}}\quad
		\mathrm{in}\,\,C^{1, \alpha}_{loc}(\Omega\setminus \partial \mathfrak{E}_*)$$ 
		for some $\alpha= \alpha(n, s)\in (0, 1)$, where 
		$$\partial \mathfrak{E}_*:=\bigcup^m_{j=1}\partial E_{*_j}.$$
		
		(iii)\, For each open $\Omega'$ compactly contained in $\Omega$, we have
		$$\mathcal{E}(u_k, \Omega')\to \frac{c_{n,s}}{2}\mathcal{P}^a_{2s}(\mathfrak{E}_*, \Omega').$$
		
	\end{prop}
	
	Now we are ready to formulate the corresponding inverse problem.
	
	\begin{thm}
		For $j=1, 2$, let $u^{(j)}_k\in H^s_{g^{(j)}_k}(\Omega)\cap L^p(\Omega)$ be a solution of 
		(\ref{fracAC}) corresponding to the potential $W^{(j)}$, the exterior data $g^{(j)}_k$ and the zero source. We assume that $W^{(j)}$ satisfies (\ref{WpmM}) and (\ref{pest}), and we assume that 
		$\{g^{(j)}_k\}$, $\{u^{(j)}_k\}$, $g^{(j)}$, and $\mathfrak{E}^{(j)}_*$ satisfy all the assumptions and conditions in the statement of Proposition \ref{convThM}.
		
		We further assume that $W^{(j)}$ is analytic on $\mathbb{R}$. We also further assume that 
		$g^{(j)}_k\not\equiv 1$ or $-1$ in $\mathbb{R}^n\setminus\Omega$. Let $V\subset \Omega$ be a nonempty open set.
		
		Suppose we know that
		$$u_k^{(1)}|_V= u_k^{(2)}|_V,\qquad (-\Delta)^su_k^{(1)}|_V= (-\Delta)^su_k^{(2)}|_V$$
		for all $k$. Then we have
		$$W^{(1)}= W^{(2)},\quad \partial \mathfrak{E}^{(1)}_*\cap \Omega= \partial \mathfrak{E}^{(2)}_*\cap \Omega,\quad \mathcal{P}^a_{2s}(\mathfrak{E}^{(1)}_*, \Omega')= \mathcal{P}^a_{2s}(\mathfrak{E}^{(2)}_*, \Omega')$$
		for each open $\Omega'$ compactly contained in $\Omega$.
	\end{thm}
	
	The proof relies on essentially the same argument as in the proof of Theorem \ref{Th}. To avoid redundancy, we will not go through full details again. Briefly speaking, the determination of the potential $W$ and $\{u_k\}$ in $\mathbb{R}^n$ follows from Proposition \ref{convThM} (ii), analytic continuation and the unique continuation of the fractional Laplacian. Then the determination of 
	$\partial \mathfrak{E}_*\cap \Omega$ and $\mathcal{P}^a_{2s}(\mathfrak{E}_*, \Omega')$ follows from Proposition \ref{convThM} (ii) and (iii).
	
	\bibliographystyle{plain}
	{\small\bibliography{Reference14}}
\end{document}